\documentclass[fleq]{amsart}
\usepackage{amssymb,amsmath,latexsym,amsbsy,color,enumerate,enumitem, tikz}
\numberwithin{equation}{section}

\newcommand{\comm}[1]{}
\newtheorem{theorem}{Theorem}
\newtheorem{definition}[theorem]{Definition}
\newtheorem{lemma}[theorem]{Lemma}
\newtheorem{question}[theorem]{Question}
\newtheorem{remark}[theorem]{Remark}
\newtheorem{proposition}[theorem]{Proposition}
\newtheorem{corollary}[theorem]{Corollary}
\newtheorem{example}[theorem]{Example}

\numberwithin{theorem}{section}

\theoremstyle{remark}

\DeclareMathOperator{\SO}{SO}
\DeclareMathOperator{\SL}{SL}
\DeclareMathOperator{\Sp}{Sp}

\DeclareMathOperator{\SU}{SU}
\DeclareMathOperator{\rank}{rank}

\DeclareMathOperator{\Nbhd}{Nbhd}
\DeclareMathOperator{\Cone}{Cone}

\newcommand{\CC}{{\mathbb{C}}}
\newcommand{\QQ}{{\mathbb{Q}}}
\newcommand{\RR}{{\mathbb{R}}}

\newcommand{\Rrank}{\rank_{\RR}}

\newcommand{\barr}{\overline}

\newcommand{\interior}[1]{%
  {\kern0pt#1}^{\mathrm{o}}%
}

 \newcounter{case}
 
 \renewcommand{\thecase}{\arabic{case}}

\numberwithin{equation}{section}

\begin{document}
 \title[Quasi-isometric embeddings]{Quasi-isometric embeddings of symmetric spaces and lattices: reducible case}
\author{Thang Nguyen}
    \address{Courant Institute of Mathematical Sciences, New York University, NY 10012}
 \email{tnguyen@nyu.edu}
  \email{thang.q.nguyen7@gmail.com}
\thanks{}
    \date{\today}
    \subjclass[2000]{}
    \begin{abstract}
We study quasi-isometric embeddings of symmetric spaces and non-uniform irreducible lattices in semisimple higher rank Lie groups. We show that any quasi-isometric embedding between symmetric spaces of the same rank can be decomposed into a product of quasi-isometric embeddings into irreducible symmetric spaces. We thus extend earlier rigidity results about quasi-isometric embeddings  to the setting of semisimple Lie groups. We also present some examples when the rigidity does not hold, including first examples in which every flat is mapped into multiple flats.
\end{abstract}
\maketitle

\section{Introduction}       

The coarse geometry of spaces has been capturing a lot of interest in geometric group theory lately. The quasi-isometric rigidity phenomenon is looked for in many classes of spaces and groups. We first recall the notions of quasi-isometry and quasi-isometric embedding which are used to study coarse geometry.
\begin{definition}
\label{defn:qi} Let $(X,d_X)$ and $(Y,d_Y)$ be metric spaces. Given
real numbers $L{\geq}1$ and $C{\geq}0$, a map $f:X{\rightarrow}Y$ is
called an {\em $(L,C)$-quasi-isometry} if
\begin{enumerate} \item
$\frac{1}{L}d_X(x_1,x_2)-C{\leq}d_Y(f(x_1),f(x_2)){\leq}L d_X(x_1,x_2)+C$
for all $x_1$ and $x_2$ in $X$, and, \item the $C$ neighborhood of
$f(X)$ is all of $Y$.
\end{enumerate}
\noindent If $f$ satisfies $(1)$ but not $(2)$, then $f$ is called an {\em $(L,C)$-quasi-isometric
embedding}.
\end{definition}
While the set of quasi-isometry rigidity results is quite rich by now, not many results are known for quasi-isometric embeddings. A generalization of quasi-isometry rigidity to rigidity of quasi-isometric embeddings would be parallel to the generalization of Mostow rigidity to Margulis superrigidity. In this paper, continuing from \cite{FisherWhyte18} and \cite{FisherNguyen}, we study quasi-isometric embeddings between higher rank symmetric spaces and lattices. The standing assumption in this paper is that $X$ and $Y$ are symmetric spaces, Euclidean buildings or lattices of the {\bf same $\RR$-rank} with $\RR$-rank is at least two. We also assume that the Euclidean buildings are thick.

We study quasi-isometric embeddings from $X$ into $Y$. We show that studying embeddings in this general setting can be reduced to studying embeddings between symmetric spaces and Euclidean buildings where the target space is irreducible. This means we are able to group irreducible factors of $X$ into new factors with the same ranks as irreducible factors of $Y$ and the embedding decomposes as a product of embeddings from new factors of $X$ into irreducible factors of $Y$. In particular we have the following theorem.

\begin{theorem}\label{splitting-thm1}
Let $X=X_1\times \dots \times X_n$ and $Y=Y_1\times \dots\times Y_m$ be symmetric spaces or Euclidean buildings without compact and Euclidean factors, where $X_i, Y_j$ are irreducible factors. Let $f: X\to Y$ be an $(L,C)$-quasi-isometric embedding. Then there is a decomposition $X=X_1'\times \dots X_m'$ and maps $(f_1,\dots,f_m)$ where
\begin{enumerate}
\item $X_i'$ is a product of irreducible factor of $X$ and $\Rrank(X_i')=\Rrank(Y_i)$, for $i=1,\dots, m$..
\item $f_i: X_i' \to Y_i$ is a quasi-isometric embedding for $i=1,\dots, m$.
\item $f$ is at a bounded distance from $(f_1,\dots,f_m)$.
\end{enumerate}
\end{theorem}

By using the asymptotic cone technique, Theorem \ref{splitting-thm1} is reduced to a result on bi-Lipschitz embeddings on Euclidean buildings. To state our result in that language, we first need to recall some terminology. In each flat of a Euclidean building, each Weyl subflat is parallel to the annihilator of a finite union of roots. Each finite union of roots determines the type of the subflat. We have the following definition.

\begin{definition}
An $\RR$-branched Euclidean building is a Euclidean building in which the union of Weyl hyperplanes of each type is the whole building.
\end{definition}

Simplicial trees and standard Euclidean buildings are not $\RR$-branched buildings. Their asymptotic cones and asymptotic cones of symmetric spaces are examples of $\RR$-branched Euclidean buildings. In an $\RR$-branched building, every point in a flat is contained in Weyl hyperplanes. The main reduction of Theorem \ref{splitting-thm1} is the following theorem.

\begin{theorem}\label{splitting-thm}
Let $X=X_1\times \dots \times X_n$ and $Y=Y_1\times \dots\times Y_m$ be $\RR$-branched Euclidean buildings without Euclidean factors, where $X_i$ and  $Y_j$ are irreducible factors for $1\le i \le n$ and $1\le j\le m$. Let $f: X\to Y$ be a bi-Lipschitz embedding. Then there is a decomposition $X=X_1'\times \dots X_m'$ and $f =(f_1,\dots,f_m)$ where
\begin{enumerate}
\item $X_i'$ is a product of irreducible factors of $X$ and $\Rrank(X_i')=\Rrank(Y_i)$, for $i=1,\dots, m$,
\item $f_i: X_i' \to Y_i$ is a bi-Lipschitz embedding for $i=1,\dots, m$.
\end{enumerate}
\end{theorem}

We remark that a factor $f_i$ could be chosen to be isometric if a certain assumption on the Weyl root types of $X_i'$ and $Y_i$ is satisfied. However, due to the existence of quasi-isometric embeddings that are not close to isometric embeddings (see $AN$-map in \cite[Proposition 2.1]{FisherWhyte18}), and the existence of rank one factors, we can always find an example where there is a factor $f_i$ of $f$ not uniformly close to any homothetic embedding.

Combining Theorem \ref{splitting-thm1} with results from \cite{FisherWhyte18}, we can get rigidity of the quasi-isometric embeddings in the following situation.

\begin{corollary}\label{splitting-cor}
Let $X$ and  $Y$ be symmetric spaces or Euclidean building without compact and  Euclidean factors. Assume $\Rrank (X)=\Rrank (Y)$. Furthermore, we assume that the collections of type $A$ Weyl patterns of irreducible factors of $X$ and of $Y$ are the same, with multiplicity.. Let $f:X\to Y$ be a quasi-isometric embedding. Then $f$ is at a bounded distance from a product of isometric embeddings, up to rescalings, of higher rank factors and quasi-isometric embeddings of rank one factors.
\end{corollary}

We also have a rigidity phenomenon for irreducible non-uniform lattices. First, we show that quasi-isometric embeddings of an irreducible non-uniform lattice in a semisimple, but not simple, Lie group can be reduced to the study of quasi-isometric embeddings of symmetric spaces into an irreducible space. We assume that all Lie groups in the following theorems have no Euclidean or compact factors.
\begin{theorem}\label{mainthm1}
Let $\Gamma$ be a nonuniform lattice with a trivial center in a higher rank semisimple Lie group $G$. Let $G'$ be a semisimple Lie group of the same rank as $G$ with the rank at least $2$. Assume that $G=G_1\times\dots\times G_n$ and $G'=G_1'\times \dots\times G_m'$ where $G_i$ and $G_j'$ are connected simple Lie groups for $1\le i\le n$ and $1\le j\le m$. Suppose the collections of type $A$ Weyl patterns of irreducible factors of $G$ and of $G'$ are the same, with multiplicity. Furthermore, we assume that $G'$ is not simple, i.e., $m\ge 2$. Let $\varphi: \Gamma\to G'$ be a quasi-isometric embedding. Then there exists a decomposition $G=G_1''\times\dots\times G_m''$, and embeddings $f_i: G_i''\to G_i'$ such that
\begin{enumerate}
\item $\Rrank(G_i'')=\Rrank(G_i')$,
\item $f_i$ are quasi-isometric embeddings, for all $1\le i\le m$,
\item $f=(f_1,\dots,f_m)|_\Gamma$ is at a bounded distance from $\varphi$.
\end{enumerate}
Furthermore, $f_i$ is an isometric embedding, up to rescaling, if $\Rrank(G_i')>1$.
\end{theorem}
This helps us obtain rigidity of quasi-isometric embeddings of irreducible non-uniform lattices in semi-simple groups.

\begin{theorem}\label{mainthm2}
Let $\Gamma$ and $ \Lambda$ be irreducible nonuniform lattices in higher rank semisimple Lie groups $G$ and $G'$ of the same rank and let the rank be at least $2$. Assume that the centers of $\Gamma$ and $\Lambda$ are trivial and $G$ and $G'$ have no compact or rank one factors. Furthermore, we assume:
 \begin{enumerate}
 \item The collections of type $A$ Weyl patterns of irreducible factors of $G$ and of $G'$ are the same, with multiplicity, and
\item There is no closed subgroup $G < H <  G'$ with compact $H$-orbit on $\Lambda\backslash G'$.
 \end{enumerate}
If $\varphi: \Gamma\to\Lambda$ is a QI-embedding, then $\varphi$ is at a bounded distance from a homomorphism $\Gamma' \to\Lambda$ where $\Gamma'<\Gamma$ has finite index.
\end{theorem}

Finally, we remark that quasi-isometric embeddings are much more flexible than quasi-isometries in general. Outside of the rigid situation, not much is known about quasi-isometric embeddings. We present here some examples of non-rigid quasi-isometric embeddings.

\begin{theorem}\label{non-rigid-rank1-thm}We have the following examples
\begin{enumerate}
\item There is a quasi-isometric embedding from a product of two trees into an $A_2$-building such that there is no flat maps into a neighborhood of any flat.
\item There is  a quasi-isometric embedding $H^2\times H^2\to SL(3,\CC)/ SU(3)$ such that there is no flat maps into a neighborhood of any flat.
\end{enumerate}
\end{theorem}
Our construction does not give examples of non-rigid embeddings between irreducible spaces. Thus, it is natural to ask what an embedding between irreducible spaces looks like. The only embeddings known to the author, constructed in \cite{FisherWhyte18} and above examples, are compositions of $AN$-maps. We state the following questions, see also \cite[Section 5]{FisherWhyte18}.
\begin{question}
Does there exist a quasi-isometric embedding which is not a composition of $AN$-maps?
\end{question}
We give a brief history of the quasi-isometric rigidity of symmetric spaces and lattices. One of the first and motivating result is Mostow rigidity \cite{Mostow68}, which can be formulated in term of quasi-isometry as following: a quasi-isometry of real hyperbolic spaces of dimension at least 3 or of higher rank symmetric spaces that is equivariant under an isomorphism of uniform lattices is close to an isometry. Prasad extended the result to nonuniform lattice in \cite{Prasad73}. The rigidity of other rank one symmetric spaces was obtained by Pansu \cite{Pansu89}. For higher rank symmetric spaces and buildings, Kleiner-Leeb and Eskin-Farb showed that quasi-isometries are close to isometries or homotheties in \cite{KleinerLeeb97} and \cite{EskinFarb97}. Their quasi-flat result has been one of the main tools in study of quasi-isometry of groups acting geometrically on CAT(0) spaces. For non-uniform lattices, quasi-isometry is even more rigid: first was obtained in a striking result of Schwartz\cite{Schwartz95}  for rank one, and later by Schwartz \cite{Schwartz96}, Farb-Schwartz \cite{FarbSchwartz}, Eskin \cite{Eskin98}, Drutu \cite{Drutu00} for higher rank Lie groups. Quasi-isometric embeddings of higher rank irreducible symmetric spaces were first studied by Fisher and Whyte \cite{FisherWhyte18}. In \cite{FisherNguyen}, the author and David Fisher, we studied embeddings non-uniform lattices in simple Lie groups. 

The paper is organized as follows. In Section \ref{linear:section}, we study isomorphic linear maps between Euclidean vector spaces which preserve Weyl patterns. The key fact we obtain from this section is that such linear maps send irreducible factors into irreducible factors. In Section \ref{Decomposition-section}, we extend the claim that irreducible factors map into irreducible factors to hold true for bi-Lipschitz embeddings between $\RR$-branched Euclidean buildings.

We prove Theorem \ref{splitting-thm}, which is the key step to obtain further rigidity results. The argument is a combination of differentiation, results from Section \ref{linear:section}, and a combinatorial argument at non-differentibility points. The rigidity of quasi-isometric embedding rigidity between symmetric spaces and buildings, Theorem \ref{splitting-thm1} and Corollary \ref{splitting-cor} when there are additional assumption on Weyl patterns also follows. Next section, we study quasi-isometric embedding of non-uniform lattices. And in the last section, we include a discussion of some non-rigid embeddings, in particular, Theorem \ref{non-rigid-rank1-thm}.

{\bf Acknowledgment:} We would like to thank Professor David Fisher for introducing the questions and for helpful discussions. Part of the work was done during the author's visit at Weizmann Institute. We would like to thank Professor Uri Bader and Weizmann Institute for the hospitality. We also want to thank Professor Robert Young for useful discussions on non-rigid embedding examples and quasi-isometric embedding in general. Finally, we thank Sandeep Bhupatiraju for his comments and help in editing the drafts.

\label{SectionIntroduction}

\section{Weyl pattern preserving linear maps}\label{linear:section}
In this section, we study linear maps that preserve Weyl patterns \cite[Definition 4.1]{FisherWhyte18}. These linear maps will be used for studying derivatives of bi-Lipschitz maps restricted to flats in asymptotic cones.
The main result in this section is the following theorem.

\begin{theorem}\label{linearmap}
Any linear map that preserves Weyl patterns maps irreducible factors into irreducible factors.
\end{theorem}
The idea is as follows. If the linear map does not send an irreducible subflat into an irreducible factor then its image splits into a product. On the other hand, an irreducible pattern, intuitively, is always more complicated than a reducible pattern, i.e., a product pattern of the same rank. This makes such linear embeddings impossible. To make this argument rigorous, we argue by induction on the dimension of the domain. We first introduce some notions and lemmas. 
\begin{definition}[Restricted pattern]
Let $V$ be a vector space with a Weyl pattern, and let $W$ be a singular subspace of $V$. The restricted pattern for $W$ is the pattern in which singular subspaces are intersections of $W$ with singular subspaces of $V$.
\end{definition}

\begin{remark}
A restricted pattern may not be linearly equivalent to any Weyl pattern of the same dimension. For example, let $V=\RR^4$ with a Weyl pattern of type $D_4$. We assume that the Weyl hyperplanes are $\{x_i=\pm x_j: 1\le i<j\le 4\}$. The restricted pattern to the hyperplane $x_1=x_2$ has 7 restricted hyperplanes which are $x_1=x_2=0, x_1=x_2=\pm x_3, x_1=x_2=\pm x_4$, and $x_1=x_2 \land x_3=\pm x_4$. Therefore this restricted pattern is not linearly equivalent to any Weyl pattern of dimension 3.
\end{remark}

Before proving Theorem \ref{linearmap}, we investigate possible restricted patterns. We study this for type $A$ patterns and type $D$ patterns. General patterns will follow from behavior of restricted patterns of type $D$. Firstly, for type $A$ patterns, we have the following lemma.

 \begin{lemma}\label{lem:Atype}
Any singular subspace of an $A_n$-pattern vector space with its restricted Weyl pattern can be linearly identified with a type $A$ Weyl pattern vector space with the same dimension as that of the subspace.
 \end{lemma}
 
 \begin{proof}
 It suffices to prove the lemma for co-dimension one singular subspaces. Consider a presentation of the $A_n$ vector space given by
 \[W=\{(x_1,\dots, x_{n+1})\in \RR^{n+1}: x_1+\dots+x_{n+1}=0\},\]
and Weyl hyperplanes
\[W_{ij}=\{(x_1,\dots, x_{n+1})\in V: x_i=x_j\},\] 
for $1\le i< j\le n+1$. It suffices to show that $W_{12}$ with restricted Weyl pattern can be linearly identified with a vector space of type $A_{n-1}$. Indeed, Weyl hyperplanes intersect $W_{12}$ and form the following restricted hyperplanes in $W_{12}$:
\[W_{ij}'=\{(x_1,\dots, x_{n+1})\in W_{12}: x_1=x_2, x_i=x_j\},\]
for $3\le i< j \le n+1$, and
\[W_{12k}=\{(x_1,\dots, x_{n+1})\in W_{12}: x_1=x_2=x_k\},\]
for $3\le k \le n+1$.

Let $U=\{(y_1,\dots, y_n): y_1+\dots+y_n=0\}$ be a vector space of type $A_{n-1}$, in which Weyl hyperplanes are $U_{ij}=\{y_i=y_j\}\subset U$, for $1\le i< j\le n$. Then the linear map $f:U\to W_{12}$, defined by $f(y_1,\dots, y_n)=(y_1,y_1, \frac{n+1}{n} y_2 -\frac 1 n y_1, \dots, \frac{n+1}{n} y_n -\frac 1 n y_1)$ maps Weyl hyperplanes $U_{ij}$ to $W_{(i+1)(j+1)}'$, for $1<i<j<n$, and $U_{1k}$ to $W_{12(k+1)}$, for $1<k\le n$. This linear map identifies $W_{12}$ along with the restricted Weyl pattern with a vector space of type $A_n$. Thus the lemma follows.
 \end{proof}

We now study restricted pattern in type $D$ vector spaces. A type $D$ pattern is the simplest pattern among patterns that are not of type $A$. It is simplest in the sense that other patterns contain type $D$ as a sub-pattern. The result for other types will follow as a corollary of the following lemma.

\begin{lemma}\label{typeD-lemma}
Let $V$ be a $D_n$ vector space, $n\ge 3$. There is a chain of singular subspaces $V_2\subset V_3 \subset \dots \subset V_{n-1} \subset V$ such that $\dim (V_i)=i$ and $V_i$ with restricted pattern contains a linearly embedded type $D_i$ pattern. Moreover, the number of restricted hyperplanes in $V_i$ is strictly greater than the number of hyperplanes in a type $D_i$ vector space, for $2\le i\le n-1$.
 \end{lemma}
 \begin{proof}
Assume that $V$ is equipped with the canonical type $D_n$ Weyl hyperplanes $V_{ij}^+=\{(x_1,\dots,x_n): x_i=x_j\}$, and $V_{ij}^-=\{(x_1,\dots,x_j): x_i=-x_j\}$, for $1\le i<j\le n$.

It suffices to show that there is a singular hyperplane with a restricted pattern containing a $D_{n-1}$ pattern. We show that $V_{12}^+$ satisfies this claim. Indeed, restricted hyperplanes of $V_{12}^+$ are $V_{12}= \{(0,0,x_3,\dots,x_n)\}$ and $V_{ij}^\pm =\{(x_2,x_2, x_3,\dots, x_n):\\ x_i=\pm x_j\}$, for $2\le i<j\le n$. The linear map $f:\RR^{n-1}\to W_{12}^+$, defined by $f(y_1,\dots,y_n)=(y_1, y_1, y_2,\dots, y_n)$, maps $\RR^{n-1}$ with the canonical $D_{n-1}$ pattern into $V_{12}^+$ where image of pattern consists of all $V_{ij}^\pm$, for $2\le i<j\le n$. We note that the restricted hyperplane $V_{12}$ is not in the collection of images of the hyperplanes of $\RR^{n-1}$ by $f$.
 \end{proof}
We make the following definition.
\begin{definition}
Let $V$ be a vector space of dimension $n$. A chain of vector subspaces $V_2\subset V_3 \subset \dots \subset V_{n-1} \subset V_n$ is called successive if $dim(V_i)=i$, for $2\le i\le n$.
\end{definition}
We have the following corollary from Lemma \ref{typeD-lemma}.
\begin{corollary}\label{othertypes-cor}
Let $V$ be a vector space of one of the following types: $E_6, E_7, E_8, F_4$, or $BC_n, D_n$ for $n\ge 3$. Then $V$ contains a successive chain of singular subflats such that each subflat with restricted pattern contains a linearly embedded type $D$ pattern of the same dimension. Moreover, the number of of restricted hyperplanes in each singular subflat is greater than the number of hyperplanes in a type $D$ vector space of the same dimension.
\end{corollary}

\begin{proof}
If $V$ is of one of the types listed, then $V$ contains a linearly embedded type $D_n$ pattern. The corollary now follows by applying Lemma \ref{typeD-lemma} repeatedly.
\end{proof} 

We conclude this section with the proof of Theorem \ref{linearmap}
 
\begin{proof}[Proof of theorem \ref{linearmap}]
Let $T:V_1\to V_2$ be a linear map that preserves Weyl patterns. It suffices to prove the theorem in the case where the pattern of $V_1$ is irreducible, i.e., we show that $T$ maps $V_1$ into an irreducible Weyl factor of $V_2$. We have that singular subspaces map to singular subspaces. Singular subspaces in the target are products of singular subspaces of irreducible factors. The claim is trivial if the rank of $V_1$ is 1. Thus, we can assume that rank of $V_1$ is at least 2.

First, we note that the claim is true when $V_1$ is of type $G_2$. Indeed, the only reducible rank 2 vector space is of type $A_1\times A_1$, which has two hyperplanes, while a type $G_2$ vector space has 6 hyperplanes. Thus, the image of a type $G_2$ plane cannot be reducible.

If $T$ does not send $V_1$ into an irreducible factor, then there exist singular subspaces $W_1\subset W_2\subset V_1$, such that $\dim(W_1)=\dim(W_2)-1$ and $T$ maps $W_1$ into an irreducible factor but does not map $W_2$ into an irreducible factor. We then have $T(W_2)\cong T(W_1)\times \RR$, and the restricted hyperplanes of $T(W_2)$ are exactly either $T(W_1)$ or products of restricted hyperplanes of $T(W_1)$ and the $\RR$ factor. In particular, the intersection of all hyperplanes but one in $T(W_2)$ is a 1-dimensional singular subspace. We have following cases based on the type of vector space $V_1$.

{\it Case 1: when $V_1$ is of type $A_n$.} Then, by Lemma \ref{lem:Atype}, the restricted patterns of $W_1$ and $W_2$ are (linearly identified) of type $A$. However, for type $A$, the intersection of all hyperplanes but one is always trivial, i.e., $\{0\}$. This contradicts the remark above.

{\it Case 2: when $V_1$ is of type $BC_n, D_n, E_6, E_7, E_8$, or $F_4$.} Suppose $dim(W_1)\ge 2$. In this case, we can choose $W_1$ and $W_2$ such that they are first two spaces of a successive chain of singular subspaces, generated by Corollary \ref{othertypes-cor}. Restricted patterns of $W_1$ and  $W_2$ contain embedded type $D$ patterns. In type $D$ vector spaces of dimension at least 3, the intersection of all hyperplanes but one is trivial. Hence, we get the same contradiction as in case 1. Therefore, we are left with the only remaining possibility: $\dim(W_1)=1$ and $\dim(W_2)=2$. In this case, $T(W_2)$ has two restricted hyperplanes. On the other hand, $W_2$ can be chosen to be the dimension 2 subspace in the successive chain obtained by Corollary \ref{othertypes-cor}. It follows that $W_2$ has at least 3 restricted hyperplanes. Thus, $T$ maps $W_2$, with at least 3 restricted hyperplanes, into $T(W_2)$ having only 2 restricted hyperplanes. This contradicts the assumption that $T$ preserves patterns.

Therefore, any pattern-preserving linear map maps each irreducible factor into an irreducible factor.
\end{proof}

\section{Decomposition of quasi-isometric embeddings}\label{Decomposition-section}

Since $D_2=A_1\times A_1$ and $D_3=A_3$, we consider them as type $A$ patterns to avoid confusion. From now on, we have the type $D_n$ only for $n\ge 4$. 

In this section, except in the proof of Theorem \ref{splitting-thm1} and Corollary \ref{splitting-cor} at the end of the section, we assume that $X=X_1\times \dots \times X_n$ and $Y=Y_1\times \dots\times Y_m$ are $\RR$-branched Euclidean buildings without Euclidean factors, where $X_i$ and $Y_j$ are irreducible factors, for $1\le i\le n$ and $1\le j\le m$. Let $f: X\to Y$ be a bi-Lipschitz embedding.
\begin{definition}[Irreducible Subflat] 
An irreducible subflat is a subflat that satisfies the following conditions
\begin{enumerate}[label=(\roman*)]
\item the subflat is contained entirely in an irreducible factor, i.e. the projection to the product of other factors is a point, and 
\item the subflat has the same rank as the rank of the irreducible factor.
\end{enumerate}
\end{definition}

By \cite[Theorem 7.2.1, Corollary 7.2.4]{KleinerLeeb97}, every bi-Lipschitz flat is contained in a finite union of flats. Moreover, we deduce the following fact.

\begin{lemma}
Each bi-Lipschitz flat is a finite union of convex polyhedra.
\end{lemma}

We note that there is a coarse version of this lemma in \cite[Lemma A.3]{Eskin98}. The idea here is the same, but the argument for bi-Lipschitz flats in asymptotic cones is much simpler, and is based on \cite{KleinerLeeb97}. We give a proof of this lemma for completeness.

\begin{proof}
Every bi-Lipschitz flat is contained in a finite union of flats. The union of flats can be written as a finite union of closed convex polyhedra. Let $\mathcal F$ be this finite family of polyhedra in the union.
Let $P$ be a polyhedron in $\mathcal F$ such that there is an interior point of $P$ belonging to a bi-Lipschitz flat $U$. Then we claim that the whole polyhedron $P$ is contained in $F$. To show this claim we show $P\cap U$ is open and closed in $P$.\\
Since $P$ and $U$ are both closed, we have $P\cap U$ is a closed set. For openness, let $x$ be interior point of $P$ such that $x\in U$. Locally at $x$, the bi-Lipschitz flat is a cone over a sphere, and thus is a finite union of sectors \cite[Corollary 6.2.3]{KleinerLeeb97}. Since $x$ is an interior point of the polyhedron, those sectors are contained in the polyhedron. In particular, there is neighborhood of $x$ contained in $U$. It follows that every interior point of $P$ is contained in $U$. Since $P\cap U$ is closed, we have $P\subset U$.\\
Therefore, every polyhedron in $\mathcal F$ either is contained in $U$ or has interiors disjoint from $U$. By discarding poyhedra having interior disjoint from $U$, we can write $U$ as a finite union of convex polyhedra.
\end{proof}

\begin{corollary}
The image of a $k$-singular subflat under a bi-Lipschitz embedding of a Euclidean building is contained in a finite union of $k$-singular subflats.
\end{corollary}

\begin{proof}
Consider two flats whose intersection is the given $k$-singular subflat. The image of each flat under the bi-Lipschitz embedding is a finite union of convex polyhedra. The image of the $k$-singular subflat is the intersection of these two finite unions. We note that bi-Lipschitz maps preserve dimension. Hence, the image of a $k$-singular subflat is contained in a finite union of $k$-singular subflats. 
\end{proof}

We now deduce the local behavior of singular subflats in the following lemma. We note that this lemma can also be found in \cite[Corollary 6.2.3]{KleinerLeeb97}, although it is stated only for maximal dimension flats there.

\begin{lemma}\label{localcone-lemma}
Let $F$ be a flat in $X$, let $M$ be a $k$-dimensional singular subflat of $F$, and let $x$ be a point in $M$. Then there is a neighborhood of $x$ in $M$ such that the image of that neighborhood is a cone with vertex $f(x)$ over a sphere of dimension $(k-1)$.
\end{lemma}

\begin{proof}
Let $F_1$ be a flat such that $M=F\cap F'$. By \cite[Corollary 6.2.3]{KleinerLeeb97}, there are neighborhoods  $U\subset F$, and $U'\subset F'$ of $x$ such that $f(U)$ and $f(U')$ are cones over spheres with vertex x. Locally at $x$, the image of $M$ is contained in $f(U)\cap f(U')$. This intersection is locally a $k$-dimensional cone. Locally at $x$, the image of $M$ is a $k$-dimensional disk. Thus, the lemma follows.
\end{proof}

For $x\in X$, and $y\in Y$, we denote the tangent cones of $X$ and  $Y$ at $x$ and $y$ by $\Sigma_xX$ and $\Sigma_yY$ respectively. We have that $\Sigma_xX$ and  $\Sigma_yY$ are spherical buildings. If $f:X\to Y$ is biLipschitz, and $f(x)=y$ then $f$ induces a map $f_*:\Sigma_xX\to \Sigma_yY$. Moreover $f_*$ is a homeomorphism onto its image. Singular subflats at $x$ correspond to walls in $\Sigma_xX$. The induced map sends each wall into a union of walls.

We introduce the following terminology.
\begin{definition} A subset is said to be contained entirely in one (irreducible) factor if projections of the subset to all factors but one are points. We also say that a subflat maps into a factor if its image is contained entirely in a factor.
\end{definition}
We briefly mention the strategy how to show a bi-Lipschitz embedding decomposes into a product. First, we show that each irreducible subflat maps entirely into a factor. Next, we show that the restrictions of the embedding on flats decompose as product maps. Finally we show the consistency of decompositions of restrictions on flats that leads us to a decomposition of the embedding on entire space.\\ 
For the first step, we need different strategies whether the rank of a subflat is one or higher. We have the following lemma for rank one irreducible subflats.
\begin{lemma}\label{rk1-not-bend-lemma}
Let $f:X\to Y$ be a bi-Lipschitz embedding between real $\RR$-branched Euclidean buildings of the same rank. Then the image of each irreducible rank one subflat is contained entirely in one irreducible factor of $Y$. Moreover, the image of two irreducible rank one subflats of the same factor are contained in the same irreducible factor of $Y$.
\end{lemma}

\begin{proof}
First, we note that the number of rank one factors of $X$ is no smaller than the number of rank one factors of $Y$. Indeed, by the argument in \cite{FisherWhyte18, FisherNguyen}, in every flat, points of differentiability are generic. At these points, derivatives preserve Weyl patterns. By Theorem \ref{linearmap}, derivatives map irreducible factors into irreducible factors. If an irreducible factor of X has rank at least two then it can not map linearly and injectively into a rank one factor of Y. Thus, the number of rank one factors of $X$ cannot be smaller than the number of rank one factors of $Y$.

Let $F$ be an arbitrary flat in $X$. By \cite[Lemma 3.1]{FisherWhyte18}, outside of a co-dimension two subset $S$ of $F$, $f$ locally maps $F$ into a flat. We first show that $f$ maps each irreducible rank one subflat of $F$, which is disjoint from $S$, into a factor. Let $F_1$ be such an irreducible rank one subflat. We have that $f(F_1)$ is a union of finitely many (possibly infinite) geodesic segments in rank one subflats. Each rank one subflat must be contained in a factor. Thus, it suffices to show that two consecutive segments belong to the same factor. Let $z$ be the common end point of two consecutive segments. Since $z$ is not in $S$, there is a flat $F'$ containing $f(z)$ such that $f$ locally maps $F$ at $z$ into $F'$ at $f(z)$. Thus, $f$ induces a homeomorphism $\Sigma_zX\cap \Sigma_zF\to \Sigma_{f(z)}Y\cap \Sigma_{f(z)}F'$, where $\Sigma_zF$ and $\Sigma_{f(z)}F$ are tangent cones of $F$ and $F'$ at $z$ and $f(z)$ respectively. We denote by $F_1^\perp$ the co-dimension one singular subflat of $F$ at $z$. This co-dimension one subflat is also the orthogonal complement of $F_1$ in $F$. Then we have that $\Sigma_{f(z)}f(F_1)\cap \Sigma_{f(z)}f(F_1^\perp)=\varnothing$. The fact that $f(F_1^\perp)$ is a union of co-dimension one subflats implies that $\Sigma_{f(z)}f(F_1^\perp)$ is a union of co-dimension one walls in $\Sigma_{f(z)}F'$. Since $\Sigma_zF_1^\perp$ separates $\Sigma_xF_1$ in $\Sigma_xF$, we have that $\Sigma_{f(z)}f(F_1^\perp)$ separates $\Sigma_{f(z)}f(F_1)$. Now, if $f$ maps $F_1$ into two different factors at $z$ then $\Sigma_{f(z)}f(F_1)$ is a union of two points which is not in a joint factor of $\Sigma_{f(z)}F'$. It follows that these two points are connected by an arc of length $\frac \pi 2$. This arc does not contain any singular points in its interior, and hence cannot be separated by a union of walls. This is a contradiction. Thus, $f(F_1)$ does not change factors at $f(z)$. Therefore, $f(F_1)$ is contained entirely in a factor.

We now show that $f$ maps all irreducible rank one subflats in $F$, which are parallel to $F_1$, entirely into a factor. Every irreducible rank one subflat in $F$, that is disjoint from $S$ maps entirely into a factor by the above argument. Moreover, if the subflat is parallel to $F_1$ then it maps into the same factor as the image of $F_1$ since these two subflats are at finite Hausdorff distance apart. Let $\barr{F_1}$ be an arbitrary subflat parallel to $F_1$ in $F$. Since $S$ is of co-dimension two, $\barr{F_1}$ can be approximated by a sequence of subflats that are disjoint from $S$. The projections of  image every subflat in the sequence to all factors but one are just points. It follows that $f(\barr{F_1})$  has projections to all but one factors are just points, i.e., $f(\barr{F_1})$ is contained entirely in a factor.

Therefore, $f$ maps each irreducible rank one subflat in $F$ entirely into one irreducible factor of $Y$. The claim for general irreducible rank one subflats in $X$ follows from the fact that given two arbitrary flats in $X$, there is a finite sequence of flats in $X$ with two given flats being the first and last flats in the sequence and two consecutive flats in the sequence intersect in a half-flat.
\end{proof}

We can now decompose a bi-Lipschitz embedding as a product of embeddings into irreducible factors by the following theorem. The claim that each irreducible higher rank subflats maps entirely into a factor is also contained in the proof of the theorem.
\begin{theorem}\label{splitting-thm}
Let $X=X_1\times \dots \times X_n$ and $Y=Y_1\times \dots\times Y_m$ be $\RR$-branched Euclidean buildings without Euclidean factors, where $X_i, Y_j$ are irreducible factors. Let $f: X\to Y$ be a bi-Lipschitz embedding. There is a decomposition $X=X_1'\times \dots X_m'$ and $f =(f_1,\dots,f_m)$ where
\begin{enumerate}
\item $X_i'$ is a product of irreducible factors of $X$ and $\Rrank(X_i')=\Rrank(Y_i)$, for $i=1,\dots, m$, and
\item $f_i: X_i' \to Y_i$ is a bi-Lipschitz embedding for $i=1,\dots, m$.
\end{enumerate}
\end{theorem}
\begin{proof}

Let $F$ be a flat. We write $F=F_1\times \dots \times F_n$, a product of subflats of irreducible factors. By \cite[Lemma 3.1]{FisherWhyte18}, there is a set $\Sigma\subset F$ that is the complement of a co-dimension two subset, such that on $\Sigma$, $f$ locally maps $F$ to a flat. Since $f$ is bi-Lipschitz, it is differentiable a.e. on $\Sigma$. At a point of differentiability, the derivative $Df$ is a linear map preserving Weyl patterns. By Theorem \ref{linearmap}, $Df$ maps irreducible factors into irreducible factors. Thus, $f$ locally maps each irreducible subflat, at points of differentiability, into an irreducible subflat.

Consider an irreducible subflat of the form $F_i\times \{u\}\subset F$, such that almost every point in $F_i\times \{u\}$ is in $\Sigma$ and $f$ is differentiable a.e. on $F_i\times \{u\}$. By Lemma \ref{rk1-not-bend-lemma},  if the rank of $F_i$ is one, then $f(F_i\times \{u\})$ is contained entirely in one irreducible factor. We show that this is also true in the case where the rank of $F_i$ is greater than one.\\
Let the rank of $F_i$ be at least two. We have that $f(F_i\times \{u\})$ is contained in a union of finitely many subflats of the same dimension. Locally, at points of differentiability, $f$ maps $F_i\times \{u\}$ into a factor of $Y$. Since the set of points of differentiability of $f|_F$ is dense on $F_i\times \{u\}$, and at these points the subflats are contained in irreducible factors, we conclude that $f(F_i\times \{u\})$ is contained in a finite union of subflats in which each subflat in the union is contained entirely in a factor. Now, consider an arbitrary point $(w,u)\in F_i\times\{u\}$. By Lemma \ref{localcone-lemma}, there is a neighborhood of $(w,u)$ in $F_i\times \{u\}$ such that the image of the neighborhood is a cone over a sphere. Since $f(F_i\times \{u\})$ is contained in a finite union of subflats in which each subflat is contained entirely in a factor, this sphere is contained in the finite union of irreducible subflats. Furthermore, because rank of $F_i$ is at least two, this sphere is connected. Thus, it is contained entirely in one irreducible factor. It follows that the neighborhood of $(w,u)$ maps into one irreducible factor. Therefore, locally at any point, $F_i\times \{u\}$ maps into one factor. Hence, connectedness implies that the whole subflat $F_i\times \{u\}$ maps into one factor.

Let $F_i\times \{u_1\}$ and $F_i\times \{u_2\}$ be two parallel subflats such that almost every point in $F_i\times \{u_1\} \cup F_i\times \{u_2\}$ is in $\Sigma$ and $f$ is differentiable a.e. on $F_i\times \{u_1\} \cup F_i\times \{u_2\}$. Then each of the subflats $F_i\times \{u_1\}$ and $F_i\times \{u_2\}$ maps into one factor. Since two subflats have finite Hausdorff distance from each other, we conclude that they map into the same factor. Furthermore, any subflat $F_i\times \{u\}$ can be approximated by a sequence of parallel subflats $F_i\times \{u_n\}$, where almost every point in $F_i\times \{u_n\}$ is not in $\Sigma$, and is a point of differentiability. Every subflat $F_i\times \{u_n\}$ maps into one irreducible factor, for all $n$. Hence, $F_i\times \{u\}$ maps into the same factor as $F_i\times \{u_n\}$ does.

We can decompose, and possibly re-arrange, $X=X_1'\times \dots \times X_m'$ and $Y=Y_1\times\dots\times Y_m$ such that for all $1\le i\le m$,
\begin{itemize}
\item $Y_i$ is irreducible,
\item $X_i'= X_{i_1}\times \dots \times X_{i_{k(i)}}$ is a product of irreducible factors, with $\Rrank(X_i')= \\ \Rrank(Y_i)$,
\item if we write $F=F_1'\times \dots \times F_m'$, where $F_i'\subset X_i'$, then $f$ maps the singular subflat $F_i'$ into the factor $Y_i$.
\end{itemize}

We now show that every $X_i'$-subflat maps into the factor $Y_i$. Let $F_i''$ be an $X_i'$-subflat such that $F_i''$ intersects $F_i'$ in a half-subflat. Consequently, the flat $F'=F_1'\times\dots\times F_{i-1}'\times F_i''\times F_{i+1}'\times \dots F_m'$ intersects $F$ in a half-flat. We know that all parallel $X_i'$-factor subflats of $F$ map into a factor. The same is true for parallel $X_i'$-factor subflats of $F'$. Since there are $X_i'$-factor subflats of $F$ and $X_i'$-factor subflats of $F'$ intersecting in half-subflats, it follows easily that $X_i'$-factor subflats of $F$ and $X_i'$-factor subflats of $F'$ map into a same factor. Note that any two flats in a building can be connected by a sequence of flats in which any two consecutive flats intersect in half-flats. Therefore, we conclude that all $X_i'$-factor subflats of $X$ map into the same factor. Since $X_i'$-subflats of $F$ maps into the factor $Y_i$, we have that every $X_i'$-subflat in $X$ maps into the factor $Y_i$.

Next, we show that the map $f$ can be decomposed as a product map. Because $X_i'$-subflats map into the factor $Y_i$, we have that for $x_i\in X_i'$ ($i=1,\dots,m$) there exist $y_j(x_1,\dots,x_{i-1},F_i',x_{i+1},\dots,x_m)\in Y_j$ for $j\neq i$ such that \\$proj_{Y_j}(f(\{x_1\}\times \dots\times \{x_{i-1}\}\times F_i'\times \{x_{i+1}\}\times \dots \times \{x_m\})) = y_j$. On the other hand, if $F_i''$ is a flat in $X_i'$, which has nonempty intersection with $F_i'$, then\\ $proj_{Y_j}(f(\{x_1\}\times \dots\times \{x_{i-1}\}\times F_i''\times \{x_{i+1}\}\times \dots \times \{x_m\})) = y_j$. This is because the projection of the image of a $X_i'$-subflat to $Y_j$ factor is a point, and thus the projections of the images of two $X_i'$-subflats, with nonempty intersection on $Y_j$ factors, coincide. It follows that $y_j$ does not depend on the $X_i'$-coordinate. In other words, $y_j$ depends only on the $X_j'$-coordinate. Therefore, $f$ can be decomposed as a product map, i.e., there are map $f_i:X_i'\to Y_i$ for $i=1,\dots,m$ such that $f=(f_1,\dots,f_m)$.

In conclusion, after possibly re-arranging and re-indexing, we can decompose\\ $X=X_1'\times\dots\times X_m'$ and $f=(f_1,\dots, f_m)$ such that $f_i: X_i\to Y_i$ are bi-Lipschitz embeddings for $1\le i\le m$.
\end{proof}

We can now decompose general quasi-isometric embeddings between symmetric spaces and Euclidean buildings into products of embeddings into irreducible targets. The idea is to use the asymptotic cone argument, we can transfer the results for bi-Lipschitz embeddings between $\RR$-branched buildings into results for symmetric spaces and Euclidean buildings.
\begin{proof}[Proof of Theorem \ref{splitting-thm1}]
First we make the following claim: there is a constant $D(L,\\C,X,Y)$ such that for any subflat of the form $F_1\times \{u\}$, there exists an $i$ such that $1\le i\le m$ and $diam(proj_j (f(F_1\times \{u\}))< D $ for all $j\neq i$. Suppose this claim is not true, then there are irreducible $X_1$-subflats $S_n$ such that the projections of images of the subflats $S_n$ onto at least two factors have diameters tending to infinity. We choose a sequence of numbers $(c_n)$ such that $c_n$ is the minimum of $n$ and two diameters of projections of $f(S_n)$ onto two factors. It follows that the induced bi-Lipschitz map on asymptotic cones with respect to the sequence of rescalings $(c_n)$ and the sequence of centers on $(S_n)$, does not map an irreducible subflat entirely into an irreducible factor. This contradicts Theorem \ref{splitting-thm}. Moreover, the argument applies not only to $f$ but to all $(L,C)$-quasi-isometric embeddings. Hence the constant $D$ depends only on the constants $L,C$ and spaces $X$ and $Y$.

By the same argument as in the proof of Theorem \ref{splitting-thm}, we can decompose, possibly after rearranging, $X=X_1'\times \dots\times X_m'$ such that projections of the images of subflats of form $F_i'\times \{*\}$ onto factors $Y_j$ have finite diameters for all $j\neq i$. Thus, there are maps $f_i: X_i'\to Y_i$ for $i=1,\dots, m$, such that $f$ is at a bounded distance from the product map $(f_1,\dots, f_m)$. The fact that $f_i$ are quasi-isometric embeddings, for all $1\le i\le m$, follows from the fact that $f|_{X_i'}$ is a quasi-isometric embedding uniformly close to $f_i$.
\end{proof}

	Before proving the rigidity result in Corollary \ref{splitting-cor}, we recall two important facts from the classification of Weyl pattern preserving linear maps in \cite[Section 4]{FisherWhyte18}. Firstly, there is no Weyl pattern preserving linear map from a domain which is not of type $A$ domain into an irreducible target space of type $A$\cite[Corollary 4.10]{FisherWhyte18}. Secondly, a Weyl pattern preserving linear map into an irreducible space can be non-conformal only if the domain has $A$-type factor while the target is either not of type $A$ or of type $A$ with a different rank \cite[Lemma 4.6, Lemma 4.7, Corollary 4.10]{FisherWhyte18}.

\begin{proof}[Proof of Corollary \ref{splitting-cor}]
By Theorem \ref{splitting-thm1}, $f$ splits into a product of quasi-isometric embeddings. We need to show that each embedding factor is at a bounded distance from an isometric embedding, up to rescaling. Note that --$f$ induces a bi-Lipschitz map on asymptotic cones $[f]:[X]\to [Y]$. By Theorem \ref{splitting-thm1}, we have a decomposition $X=X_1'\times\dots\times X_m'$ such that $f_i:X_i'\to Y_i$ is a quasi-isometric embedding for $1\le i\le m$. We look at asymptotic cones and the induced map $[f]=([f_1],\dots,[f_m])$. Consider one factor map. We note that if we have a pattern preserving linear map into a type $A$ pattern, then the domain must have at least one factor of type $A$. Since the collection of type $A$ factors of the domain and the target are the same, if the target of a factor map is of type $A$ then the domain is also of type $A$ with the same rank. The assumption that $X$ and $Y$ have the same collection of type $A$ patterns (with multiplicity) also implies that any pattern preserving linear factor map from $[X_i']$ into $[Y_i]$ is conformal. Hence, by \cite[Theorem 1.8]{FisherWhyte18}, $f_i$ is at a bounded distance from an isometric embedding, up to rescaling.
\end{proof}

\section{Embedding of irreducible non-uniform lattices}
In this section, we study quasi-isometric embeddings of non-uniform lattices. We assume the setting of Theorem \ref{mainthm1}. Given an embedding of a non-uniform lattice, we get an embedding of a neighborhood of the lattice. However, the composition of this map and nearest point projection does not give a quasi-isometric embedding of the whole symmetric space. Thus, we can not directly use Corollary \ref{splitting-cor} from section \ref{Decomposition-section} and \cite[Theorem 1.4]{FisherNguyen} to obtain the rigidity of the embedding. Instead, we combine the ideas and proofs from section \ref{Decomposition-section} and arguments in \cite{FisherNguyen} to prove the result. The key point is that the irreducibility of the lattice ensure that there are enough flats and diverging patterns to run argument. We quote the results from \cite{FisherNguyen} which do not rely on Lie groups being simple and we refer to that paper for a detailed proof.

Let $\varphi: \Gamma\to \Lambda$ be a quasi-isometric embedding of irreducible non-uniform lattices. Then $\varphi$ induces a map $X\to Y$ by pre-composing with a closest point projection of $X$ to $\Gamma$ (pick one if there are more than one closest points) and post-composing with the inclusion $\Lambda\to Y$. We note that an irreducible higher rank lattice is quasi-isometrically embedded into the symmetric space by  \cite{LMR}. Abusing the notation, we still denote by $\varphi: X\to Y$ the induced map. This induced map is not a quasi-isometric embedding. However, an argument using ergodicity shows that on many flats, the induced map behaves almost like a quasi-isometric embedding.

We assume that the semisimple groups satisfy the assumptions of Theorem \ref{mainthm1}. The following proposition is obtained from Fisher-Nguyen's \cite{FisherNguyen}. We note that the assumption that $\Gamma$ is irreducible makes the Howe-Moore ergodicity theorem applicable.

\begin{proposition}\cite[Section 3.1]{FisherNguyen}\label{prop:sublinearflat}
Let $\delta >0$ and let $\omega$ be a non-principal ultrafilter. There exists a non-increasing function $\theta_\delta$ converging to $0$ and a full measure family $\mathcal F$ of sub-$\theta_\delta$-diverging flats in $X$ such that for any sequence of flats $F_n$ that are sub-$\theta_\delta$-diverging w.r.t. $x_n$, and any sequence $c_n$ with $\lim\limits_\omega c_n=+\infty$, $\varphi$ induces a bi-Lipschitz map $[\varphi]|_{[F_n]}$from $[F_n]$ into $\Cone(Y_n, \varphi(x_n),c_n,\omega)$. The claim still holds true for sequences of unions of finitely many intersecting flats. Furthermore, for any union of finitely many hyperplanes in $[F_n]$, we could choose a union of finitely many flats intersecting $[F_n]$ in the union of hyperplanes.\qed
\end{proposition}

The definition of a sub-$\theta_\delta$-diverging flat is technical as it makes repeated use of the ergodic theorem. We refer to  \cite[Section 3.1]{FisherNguyen} for a precise definition. Roughly speaking, a flat is sub-$\theta_\delta$-diverging with respect to a fixed point if the proportion of the measure of the subset of points in the ball of radius $r$ around the fixed point of the distance at least $\theta_\delta(r)r$ away from the lattice $\Gamma$ goes to zero as $r$ goes to infinity.

\begin{proposition}\label{flat2flat-in-cones}
For $F_n\in \mathcal F$ as in Proposition \ref{prop:sublinearflat}, $[\varphi]|_{[F_n]}$ is a product map into $\Cone(Y, \varphi(x_n),c_n,\omega)$, and $[\varphi]|_{[F_n]}([F_n])$ is a flat. Moreover, each factor of $[\varphi]|_{[F_n]}$, up to a rescaling, is a map given by an element in the Weyl group.
\end{proposition}

\begin{proof}
The argument here is a combination of the argument in \cite[Section 3.1-3.2]{FisherNguyen} and the argument in Section \ref{Decomposition-section}. To simplify notation, we denote by $[Y]$ the asymptotic cone $\Cone(Y, \varphi(x_n),c_n,\omega)$.

We have that $[\varphi]|_{[F_n]}([F_n])$ is a bi-Lipschitz flat and thus is contained in a union of finitely many flats in $[Y]$. By picking other sequences of flats that have limits intersecting $[F_n]$ in hyperplanes, we conclude that the subflats of $[F_n]$ map into unions of finitely many subflats in $[Y]$. In $[F_n]$, outside of a co-dimension 2 subset, $[\varphi]|_{[F_n]}$ locally maps a flat to a flat. Using the same argument as in Section \ref{Decomposition-section}, we conclude that $[\varphi]|_{[F_n]}$ is decomposed as a product map. Now, the domain and the range of each factor map satisfy the condition that all linear maps from the Weyl pattern of the domain to the Weyl pattern of the range is conformal. Hence, each factor map sends flats to flats. It follows that $[\varphi]|_{[F_n]}([F_n])$ is a flat in $[Y]$.
\end{proof}

Applying Proposition \ref{flat2flat-in-cones} to the constant sequence consisting of a fixed flat $F\in \mathcal F$, we get that there is a decomposition $G=G_1''\times\dots\times G_m''$ such that each $G_i''$-subflat of $[F]$ maps into a $G_i'$-subflat of $[Y]$. We show that this decomposition of $G$ does not depend on the choice of $F\in \mathcal F$. We note that if $E\in \mathcal F$ such that $[F]$ intersects $[E]$ in either a half-flat, or a Weyl chamber, or a hyperplane, then the decompositions of $G$ obtained from Proposition \ref{flat2flat-in-cones}, with respect to the choice of flat $[F]$ and  the choice of flat $[E]$, agree. As $\mathcal F$ has full measure in the set of flats, any two flats in $\mathcal F$ can be connected by a chain of flats in $\mathcal F$ such that any two consecutive flats coarsely intersect in either a half-flat, a Weyl chamber, or a hyperplane. Therefore the decomposition is independent of $F\in \mathcal F$.

Now, we apply arguments in \cite[Section 3, Section 4]{FisherNguyen} to obtain the following proposition (which is a combination of Proposition 3.3, Corollary 3.8, Proposition 3.9 and Proposition 4.1 in \cite{FisherNguyen}). We note that $\Gamma$ is irreducible, hence all abelian subgroups corresponding to subflats acting ergodically on $\Gamma\backslash G$.

\begin{proposition}\label{image_of_flat:prop}
Let $F$ be a sub-$\theta_\delta$-diverging flat w.r.t. $x$. Then there is a flat $F'\subset Y$ such that all of the followings hold true
\begin{enumerate}
\item $\varphi(x)$ is at a bounded distance from $F'$ and $\varphi(F)$ is sub-linearly divergent from $F'$,
\item images of Weyl chambers with the vertex at $x$ are sub-linearly divergent from a finite union of (a fixed number of) Weyl chambers in $Y$ with the vertex at $\varphi(x)$.
\item images of a subflat in $F$ containing $x$ is sublinearly divergent from a subflat in $F'$ containing $\varphi(x)$.
\item The set of $y\in F$ such that $F$ is a sub-$\theta_\delta$-diverging flat w.r.t. $y$ has a large proportion of Lebesgue measure in $F$ and the images of such points are uniformly close to $F'$.
\end{enumerate}
\end{proposition}

\begin{proof}See \cite[Proposition 3.3, Corollary 3.8, Proposition 3.9, and Proposition 4.1]{FisherNguyen}.
\end{proof}

Now we are ready to prove Theorem \ref{mainthm1}.
\begin{proof}[Proof of Theorem \ref{mainthm1}]
Let $F\in \mathcal F$, and let $F'$ be the flat is obtained by applying Proposition \ref{image_of_flat:prop}. To simplify the exposition, we abuse notations and write $F'=\varphi(F)$. 

We decompose $G=G''_1\times\dots\times G''_m$ such that for any $F\in \mathcal F$, $[\varphi]|_{[F]}$ maps each $G''_i$-subflat of $[F]$ to a $G'_i$-subflat of $[\varphi(F)]$. Let $X_i'$ and $Y_i$ be symmetric spaces associated to $G_i''$ and $G_i'$.

Let $x\in F$ be a point such that $F$ and all $F_\alpha$ for $\alpha\in \Xi$ are sub-$\theta_\delta$-diverging w.r.t.~ $x$ and let $F'=\varphi(F)$, $F'_\alpha=\varphi[F_\alpha]$. We recall that the map $\varphi$ now is the composition of the original $\varphi$ and nearest point projection. If $F\in \mathcal F$ is sub-$\theta_\delta$-diverging w.r.t.~$x$ then $x$ is $R$-close ($R$ depends on $\Gamma$, $G$, and $\delta$) to $\Gamma$, hence $\varphi(x)$ is well defined up to some $(2LR+C)$-distance, i.e., independent of the choice of a closest point in the projection. We also recall that there is a subset $\Omega_\delta\subset \Gamma\backslash G$ which is contained in an $R$-neighborhood of $p(1)\in \Gamma\backslash G$, such that a flat $F\in \mathcal F$ is sub-$\theta_\delta$-diverging w.r.t.~ $x\in F$ if we can write $x=\pi(g)$ and $F=\pi(gA)$ for some $g\in \Gamma\Omega_\delta$ (See \cite[Section 3]{FisherNguyen}). We have the following lemma.

\begin{lemma}\label{proj-estimate-lemma}
If $g, g'\in \Gamma\Nbhd_1(\Omega_\delta)$ such that $d(proj_{G_i''}(g),proj_{G_i''}(g'))<1$ then\\ $d(proj_{G_i'}(\varphi(g)),proj_{G_i'}(\varphi(g')))<12LR+4D+4L+3C$.
\end{lemma}

We will prove this lemma after finishing the proof of Theorem \ref{mainthm1}. We note that $proj_{G_i''}(\Gamma\Nbhd_1(\Omega_\delta))=G_i''$ for all $1\le i\le m$ because $\Gamma$ is irreducible. We define a map $\varphi_i: G_i''\to G_i'$ as follow: for every $g_i\in G_i''$, we pick $g\in \Gamma\Nbhd_1(\Omega_\delta)$ such that $proj_{G_i''}(g)=g_i$. We define $f_i(g_i)=proj_{G_i'}(\varphi(g))$. By Lemma \ref{proj-estimate-lemma}, $\varphi_i$ is well-defined up a finite distance.

We need to show that $\varphi_i$ is a quasi-isometric embedding from $G_i''$ into $G_i'$. Indeed, for $g_{i}$ and $g_i'$ in $G_i''$, it follows from the ergodicity of the action of $\widehat{G_i''}=G_1''\times\dots\times G_{i-1}''\\\times G_{i+1}''\times \dots G_m''$ on  $\Gamma\backslash G$, that there are $g$ and $g'$ in $\Gamma\Omega_\delta$ such that \[d(proj_{G_i''}(g),g_i)<1, d(proj_{G_i''}(g'),g_i')<1,\] and \[d(proj_{\widehat{G_i''}}(g),proj_{\widehat{G_i''}}(g'))<1.\] Hence 
\[|d(g,g')-d(g_i,g_i')|<2.\]
On the other hand, there is a flat $F\in \mathcal F$ intersecting an $R$-neighborhoods of $g$ and $g'$ in $\Gamma\Omega_\delta$. To simplify notations, let us denote $C'=4(12LR+4D+4L+3C)$. By the triangle inequality and because $\varphi$ is a quasi-isometric embedding, we have that $$d(proj_{\widehat{G_i'}}(\varphi(g)),proj_{\widehat{G_i'}}(\varphi(g')))<C'.$$It follows that 
\[d(\varphi(g),\varphi(g'))-C'<d(proj_{{G_i'}}(\varphi(g)),proj_{{G_i'}}(\varphi(g')))=d(f_i(g_i),f_i(g_i'))<d(\varphi(g),\varphi(g')).\]
Since $\varphi$ is a quasi-isometric embedding on a neighborhood of $\Gamma$, it follows that $f_i$ is also a quasi-isometric embedding.\\
In the case that $G_i'$ has higher rank, by \cite{FisherWhyte18} $f_i$ is at a bounded distance from an isometric embedding. In that case, we can replace $f_i$ by that isometric embedding.

Finally, we show that the original quasi-isometric embedding is uniformly close to the product map $(f_1,\dots,f_m)$. Let $\gamma=(\gamma_1,\dots, \gamma_m)\in \Gamma$, there is $g=(g_1,\dots,g_m)$ which is in an $R$-neighborhood of $\gamma$ in $\Omega_\delta$. We have that \[L^{-1}R-C<d(\varphi(\gamma),\varphi(g))<LR+C,\]and \[L^{-1}R-C'-C<d(f_i(g_i),f_i(\gamma_i))<LR+C'+C.\]
As $\varphi$ and $(f_1,\dots,f_m)$ are uniformly close on $\Gamma\Omega_\delta$, it follows that $\varphi$ and $(f_1,\dots,f_m)$ are uniformly close on $\Gamma$.
\end{proof}
Now we prove Lemma \ref{proj-estimate-lemma}.
\begin{proof}[Proof of Lemma \ref{proj-estimate-lemma}]
Let $\gamma$  and$ \gamma'$ be elements of $ \Gamma$ such that $g\in \gamma \Nbhd_1(\Omega_\delta)$ and $g'\in \gamma'\Nbhd_1(\Omega_\delta)$. Since the family $\mathcal F$ has full measure in the set of flats, there is $t\in G$ such that $\pi(tA)$ is a flat in $\mathcal F$ and $tA$ has non-empty intersections with both $\gamma\Omega_\delta$ and $\gamma'\Omega_\delta$. Let $\tilde{g}\in tA\cap \gamma\Omega_\delta$ and $\tilde{g}'\in tA\cap \gamma'\Omega_\delta$. Denote by $A_i$ the $G_i''$-factor abelian subgroup of $A$. The two subflats $\pi(\tilde{g}A_i)$ and $\pi(\tilde{g}'A_i)$ are parallel in the flat $\pi(tA)$. The Hausdorff distance between the two subflats $\pi(\tilde{g}A_i)$ and $\pi(\tilde{g}'A_i)$ is bounded by $2R$. Since $\Gamma$ is irreducible, $A_i$ acts ergodically on $\Gamma\backslash G$. By the definition of $\Omega_\delta$, it follows that the sets $\tilde{g}A_i$ and $\tilde{g}'A_i$ contain a large portion of points belonging to $\Gamma\Omega_\delta$. Therefore, there are $h\in \tilde{g}A_j\cap \Gamma\Omega_\delta$ and $h'\in \tilde{g'}A_j\cap \Gamma\Omega_\delta$ such that $d(h,h')<4R$. It follows that $d(\varphi(h),\varphi(h'))<6LR+C$. On the other hand, because $\tilde{g},\tilde{g}'\in \Gamma\Omega_\delta$, there are $G_i'$-subflats $SF_i$ and $SF_i'$ in $Y$ such that the image under $\varphi$ of $\tilde{g}A_i\cap \Gamma\Omega_\delta$ and $\tilde{g}'A_i\cap \Gamma\Omega_\delta$ are $D$-close to $SF_i$ and $SF_i'$ respectively. We note that $SF_i$ and $SF_i'$ are $G_i'$-subflats, thus the projections of each of them to $G_i'$ factors are just points. Hence $d(proj_{G_i'}(\varphi(\tilde{g})),proj_{G_i'}(\varphi(h)))<2D$ and $d(proj_{G_i'}(\varphi(\tilde{g}')),proj_{G_i'}(\varphi(h')))<2D$. In conclusion, applying the triangle inequality we obtain
\[d(proj_{G_i'}(\varphi(g)),proj_{G_i'}(\varphi(g')))<12LR+4D+4L+3C.\]\end{proof}
 
\begin{proof}[Proof of Theorem \ref{mainthm2}]
When $G'$ is simple, the theorem is proved in \cite[Theorem 1.4]{FisherNguyen}. We assume that $G'$ is not simple. By Theorem \ref{mainthm1}, $\varphi$ is uniformly close to a product of isometric embeddings. By \cite[Section 6]{FisherNguyen}, we have that $\varphi$ is uniformly close to a group homomorphism from a finite index subgroup of $\Gamma$ into (possibly a conjugate by a commensurator of) $\Lambda$.
\end{proof}
\begin{remark}
Theorem \ref{mainthm2} can also be proved by the general scheme: first show good flats map to flats. Next, use \cite[Section 5]{FisherNguyen} to prove the regularity of the boundary map and obtain geometric rigidity. Finally, use \cite[Section 6]{FisherNguyen} to obtain the algebraic conclusion. In the case when $G$ is not simple, the argument presented here is simpler as it circumvents \cite[Section 5]{FisherNguyen}, which is the most difficult part in the case when $G$ is simple.
\end{remark}

\section{Non-rigid quasi-isometric embeddings}

\subsection{Combinatorial interpretation of $AN$-maps}In this subsection we give a way of interpreting $AN$-maps which does not use any algebraic structure. This is helpful in constructing quasi-isometric embeddings between Euclidean buildings.

We recall some definitions which are used in \cite{Leeb00}. Let $X$ be a symmetric space or Euclidean building. Then $\partial X$ is a spherical building.

\begin{definition}[Parallel Set and Cross Section]\cite[Definition 3.4]{Leeb00}
Given a singular unit sphere $s\subset \partial X$, the {\bf parallel set} $P(s)$ of $s$ is the union of flats or subflats which have $s$ as their boundaries at infinity. The parallel set is isometric to a product of a Euclidean space and a symmetric space or Euclidean building as
\[P(s)=\RR^{dim(s)+1}\times CS(s).\]
The symmetric space or Euclidean building $CS(s)$ is called the {\bf cross section} of $s$.
\end{definition}
We can also define a cross section of a simplicial cell in $\partial X$ as follow. Let $s\subset \partial X$ be a singular unit sphere. Let $c\subset s$ be a simplicial cell such that $\dim (c)=\dim (s)$. We define the cross section of $c$, denoted by $CS(c)$, to be the cross section of $s$. By \cite[Lemma 3.5]{Leeb00}, this definition is well-defined.
\begin{example}Let $X$ be the $A_2$-building associated with $\SL(3, \QQ_p)$. For any pair of opposite singular 0-cells $\zeta_-$ and $\zeta_+$ on $\partial X$, we then have that $CS(\zeta_-)$ and $CS(\{\zeta_-,\zeta_+\})$ are homothetic to $T_{p+1}$, where $T_{p+1}$ is the $(p+1)$-regular simplicial tree. Moreover, cross sections of all singular points in $\partial X$ are isometric.
\end{example}
\begin{example}
Let $X=\SL(3,\RR)/\SO(3)$, a symmetric space of type $A_2$. In $X$ there are many totally geodesic copies of the hyperbolic plane $H^2$. Each copy of $H^2$ can be identified with the cross section of two boundary points of a singular geodesic.
\end{example}

We recall the definition of spaces of strong asymptote classes from \cite[Section 2.1.3]{Leeb00}. For a point $\xi\in \partial X$, we consider the set of geodesic rays asymptotic to $\xi$. The asymptotic distance between two rays $\rho_1,\rho_2: [0,\infty)\to X$ is given by
\[d_\xi(\rho_1,\rho_2)=\inf\limits_{t_1,t_2\to\infty}d(\rho_1(t_1),\rho_2(t_2)).\]
The space of strong asymptote classes associated to $\xi$, denoted by $X_\xi$ is the space of geodesic rays asymptotic to $\xi$ quotiented out by the equivalence relation defined by zero $d_\xi$-distance. There is a natural projection $proj_\xi: X\to X_\xi$ by mapping each $x\in X$ to the asymptote class of the geodesic ray from $x$ to $\xi$. We note that this projection is distance non-increasing map. Moreover, if $\xi$ is an interior point of a $k$-dimensional cell $c\subset \partial X$, then $X_\xi$ is isometric to $\mathbb{R}^k\times CS(c)$. In particular, if $\xi$ is a singular $0$-cell then $X_\xi$ and $CS(\xi)$ are isometric.

\begin{example}\cite[Theorem 1.6]{FisherWhyte18} One of the simplest examples of an $AN$-map is a quasi-isometric embedding $f:H^2\times H^2 \to \SL(3,\RR)/SO(3)$. First, we identify $H^2$ with the set of upper triangular matrices in $\SL(2,\RR)$ and $\SL(3,\RR)/SO(3)$ with the set of upper triangular matrices in $\SL(3,\RR)$. The map can be given by an explicit formula as follows
\[f:\left(\begin{pmatrix} e^t & xe^{-t} \\0 & e^{-t} \end{pmatrix},
\begin{pmatrix} e^s & ze^{-s} \\0 & e^{-s}\end{pmatrix}\right)
\mapsto \begin{pmatrix}e^{\frac{2}{3}(t+s)} & xe^{\frac{2}{3}(s-2t)} &ze^{\frac{2}{3}(t-2s)} \\0 & e^{\frac{2}{3}(s-2t)} &0\\0 & 0 &e^{\frac{2}{3}(t-2s)}\end{pmatrix}.\]
We call a flat vertical if it is parametrized as  the set $$\left\{\left(\begin{pmatrix} e^t & xe^{-t} \\0 & e^{-t} \end{pmatrix},
\begin{pmatrix} e^s & ze^{-s} \\0 & e^{-s}\end{pmatrix}\right): t,s\in\RR\right\},$$ for some fixed $x,z\in\RR$. We note that in this example, each vertical maps to a flat, in which the vertical quadrant and its opposite map to a union of two Weyl sectors each. In particular, there is an open and dense subset of the Fursterberg boundary of $H^2\times H^2$ such that $f$ maps each chamber in the set to a union of two chambers in the Fursternberg boundary of $Y=\SL(3,\RR)/SO(3)$.

Let $\xi_1\in \partial Y$ and $\xi_2\in \partial Y$ be the endpoints of the geodesic rays $\left\{\begin{pmatrix}e^{\frac{2}{3}t} & 0 &0 \\0 & e^{\frac{-4}{3}t} &0\\0 & 0 &e^{\frac{2}{3}t}\end{pmatrix}: t\in [0,\infty)\right\}$ and $\left\{\begin{pmatrix}e^{\frac{2}{3}s} & 0 &0 \\0 & e^{\frac{2}{3}s} &0\\0 & 0 &e^{\frac{-4}{3}s}\end{pmatrix}:s\in [0,\infty)\right\}$, respectively. It can be checked that $proj_{\xi_1}\circ f$ maps each copy of $H^2$ of the form $H^2\times\{*\}\subset H^2\times H^2$ isometrically to $Y_{\xi_1}$. Similarly, $proj_{\xi_1}\circ f$ maps each copy $\{*\}\times H^2$ isometrically to $Y_{\xi_2}$.

On the other hand, if $F\subset Y$ is a flat such that $\xi_1,\xi_2\in \partial F$, then every point $x$ in $F$ is uniquely determined by $(proj_{\xi_1}(x),proj_{\xi_2}(x))$. Moreover, every point in the union of flats containing $\xi_1$ and $\xi_2$ on their boundaries is uniquely determined by projections to $Y_{\xi_1}$ and $Y_{\xi_2}$. Indeed, let $Y(\xi_1,\xi_2)$ be the union of all flats containing $\xi_1$ and $\xi_2$ on their boundaries. Let $x$ and $z$ be two arbitrary points in $Y(\xi_1,\xi_2)$. Two geodesic rays $[x,\xi_1)$ and $[x,\xi_2)$ bound a unique Euclidean sector $E_x$ which is isometric to union of two Weyl sectors. Similarly $[z,\xi_1)$ and $[z,\xi_2)$ bound a unique Euclidean sector $E_z$. If $(proj_{\xi_1}(x),proj_{\xi_2}(x))=(proj_{\xi_1}(z),proj_{\xi_2}(z))$ then $E_x=E_z$. It follows that $x=z$. Furthermore, $(proj_{\xi_1}, proj_{\xi_2}) : Y(\xi_1,\xi_2)\to Y_{\xi_1}\times Y_{\xi_2}$ is a quasi-isometry, and $Y(\xi_1,\xi_2)$ is quasi-isometrically embedded into $Y$ since $Y(\xi_1,\xi_2)$ is a union of vertical flats. Now, $f$ can be obtained by first identifying two factors $H^2$ in $H^2\times H^2$ with $Y_{\xi_1}$ and $Y_{\xi2}$ then composing it with $(proj_{\xi_1}, proj_{\xi_2})^{-1}$.
\end{example}
\begin{example}
We can now obtain a similar example of a quasi-isometric embedding $T_{p+1}\times T_{p+1}\to Y$, where $Y$ is an $A_2$ Euclidean building associated to $\SL(3,\QQ_p)$. Let $\xi_1$ and $\xi_2$ be two singular points in $\partial Y$ which are a combinatorial distance 2 apart. We have that $Y_{\xi_1}$ and $Y_{\xi_2}$ are isometric and are homothetic to $T_{p+1}$. The map $f=(proj_{\xi_1},proj_{\xi_2})^{-1}):T_{p+1}\times T_{p+1}\to Y$ is a quasi-isometric embedding. Abusing the notation, we call this embedding an $AN$-map.
\end{example}

A natural question which arises is whether all quasi-isometric embeddings are of the form of an $AN$-map. The answer is negative, as illustrated by the following example.

\begin{example}\cite[Theorem 1.7]{FisherWhyte18}\label{eg:DFW} There is an embedding $f$ from $H^2\times H^2$ into $\Sp(4,\RR)$, which is not of the form of an $AN$-map. 

Indeed, we can obtain such an embedding by composing two embeddings of the form of $AN$-maps. The $AN$-maps are $H^2\times H^2\to \SL(3,\RR)$ and $\SL(3,\RR)\to \Sp(4,\RR)$. We recall that to define an $AN$-map we have to fix a Weyl chamber at infinity to define the family of vertical flats. For the embedding $\SL(3,\RR)\to \Sp(4,\RR)$, we pick a Weyl chamber in $\partial \SL(3,\RR)$, that is disjoint from the boundary of the image of the first embedding $H^2\times H^2\to \SL(3,\RR)$, to define a vertical flat family. By examining the number of Weyl sectors that each Weyl sector maps to under the composition map, it can be shown that $f$ is not an $AN$-map.
\end{example}

We remark that in above example there is a flat mapping to a single flat. In fact, we can construct examples of quasi-isometric embeddings which are even more non-rigid: no flat maps to a finite neighborhood of a single flat. We have seen that showing flats map to flats is the first key step to obtain rigidity of quasi-isometries and quasi-isometric embeddings. Any attempt at showing that quasi-isometric embeddings are $AN$-maps may have to start with finding a family of flats that map to flats. The non-rigid examples we present below show that quasi-isometric embeddings can be wilder than what we expect naively. 
\begin{proof}[Proof of Theorem \ref{non-rigid-rank1-thm}]
Let $X$ be $T_3\times T_3$ in the first case and $H^2\times H^2$ in the second case. Let $Y$ be an $A_2$ Euclidean buildings in the first case and $\SL(3,\CC)/\SU(3)$ in the second case. Following Example \ref{eg:DFW}, the embeddings are defined by compositions of $AN$-embeddings.
\begin{enumerate}
\item Let $\xi\in\partial T_3$ and let $\varphi:T_3\to T_3$ be an quasi-isometric embedding such that $\xi$ is not in the boundary of the image. We note that the Furstenberg boundary of $T_3\times T_3$ is $\partial T_3\times \partial T_3$. We use $(\xi,\xi)$ in the Furstenberg boundary of $T_3\times T_3$ to define a family of vertical flats in $T_3\times T_3$. Let $g:T_3\times T_3\to Y$ be the $AN$-quasi-isometric embedding. We show that $f=g\circ (\varphi,\varphi)$ is the desired embedding. It is clear that $f$ is a quasi-isometric embedding as it is the composition of two quasi-isometric embeddings. We first note that under $g$, each Weyl chamber, that is opposite to $(\xi,\xi)$, maps to a union of two Weyl chambers in $\partial Y$. On the other hand, the embedding $(\varphi,\varphi)$ maps all chambers in $\partial (T_3\times T_3)$ to chambers opposite to $(\xi,\xi)$. It follows that $f$ maps each chamber in $\partial (T_3\times T_3)$ to a union of two chambers in $\partial Y$. Thus, f maps each flat in $T_3\times T_3$ to a set which is at a finite Hausdorff distance from a union of 8 Weyl sectors, which cannot be uniformly close to any single flat in $Y$.
\item We define the embedding in the same way as the previous case. Namely, pick $\xi\in \partial H^3$ and quasi-isometrically embed $H^2$ into $H^3$ avoiding $\xi$ on the boundary. Next, use $(x\i,\xi)$ to define an $AN$-map from $H^3\times H^3$ into $\SL(3,\CC)/\SU(3)$. We remark that the space of strong asymptote classes of a singular geodesic ray in $Y$ is isometric to $H^3$. The embedding is given by the composition $H^2\times H^2\to H^3\times H^3\to \SL(3,\CC)/\SU(3)$. Every flat in $H^2\times H^2$ maps to a set of finite Hausdorff distance from a union of 8 Weyl sectors, which cannot be a single flat in $Y$.
\end{enumerate}
\end{proof}
\begin{remark}We end with the following remarks.
\begin{enumerate}
\item Similar constructions work for higher rank symmetric spaces and Euclidean buildings. In particular, there are examples of quasi-isometric embeddings from product of $n$ 3-regular trees into rank $n$ Euclidean buildings for all $n\ge 2$ such that the embeddings do not map any flat into a neighborhood of a single flat.
\item A naive attempt to provide examples when the ranks and dimensions are higher, or between irreducible symmetric spaces and Euclidean buildings, does not follow easily. Furthermore, when the embeddings are between irreducible spaces, and we do not know any examples of embeddings other than compositions of $AN$-maps.
\end{enumerate}
\end{remark}

\end{document}